\documentclass[10 pt, reqno]{amsart}
\usepackage{color}
\usepackage{verbatim}
\usepackage{amssymb}
\usepackage{amscd}
\usepackage{amsmath}
\usepackage{amsthm}
\usepackage{eucal}
\usepackage{setspace}
\usepackage{url}
\usepackage[utf8]{inputenc}
\usepackage{booktabs}

\usepackage[hyperfootnotes=false, colorlinks, linkcolor={blue}, citecolor={magenta}, filecolor={blue}, urlcolor={blue}]{hyperref}

\allowdisplaybreaks

\renewcommand{\H}{\mathbb H}
\newcommand{\leftexp}[2]{{\vphantom{#2}}^{#1}%
      \kern-\scriptspace%
      {#2}}

\newcommand{\E}{{\mathcal E}}

\newcommand{\Q}{{\mathbb Q}}
\newcommand{\Z}{{\mathbb Z}}

\newcommand{\R}{{\mathbb R}}
\newcommand{\C}{{\mathbb C}}
\newcommand{\bs}{\backslash}

\newcommand{\U}{\mathcal U}

\newcommand{\SL}{{\rm SL}}
\newcommand{\SO}{{\rm SO}}

\newcommand{\Sp}{{\rm Sp}}

\newcommand{\Hom}{{\rm Hom}}

\newcommand{\vl}{{\rm vol}}

\renewcommand{\AA}{{\mathcal A}}
\newcommand{\HH}{{\mathbb H}}
\newcommand{\mat}[4]{{\setlength{\arraycolsep}{0.5mm}\left[\begin{array}{cc}#1&#2\\#3&#4\end{array}\right]}}

\newtheorem{lemma}{Lemma}[section]
\newtheorem{theorem}[lemma]{Theorem}

\newtheorem{corollary}[lemma]{Corollary}
\newtheorem{proposition}[lemma]{Proposition}

\theoremstyle{remark}
\newtheorem{remark}[lemma]{Remark}

\begin{document}

\bibliographystyle{plain}

\title[Representations and nearly holomorphic  forms]{Representations of $\SL_2(\R)$ and nearly holomorphic modular forms}

\thanks{A.S.\ is partially supported by EPSRC grant EP/L025515/1. A.P.\ and R.S.\ are supported by NSF grant DMS--$1100541$.}

\author{Ameya Pitale}
\address{Department of Mathematics
\\ University of Oklahoma\\ Norman\\
   OK 73019, USA}
\email{apitale@math.ou.edu}

\author{Abhishek Saha}
\address{Departments of Mathematics \\
  University of Bristol\\
  Bristol BS81SN \\
  UK} \email{abhishek.saha@bris.ac.uk}

\author{Ralf Schmidt}
\address{Department of Mathematics
\\ University of Oklahoma\\ Norman\\
   OK 73019, USA}
\email{rschmidt@math.ou.edu}

\begin{abstract}
In this semi-expository note, we give a new proof of a structure theorem due to Shimura for nearly holomorphic modular forms on the complex upper half plane.  Roughly speaking, the theorem says that the space of all nearly holomorphic modular forms is the direct sum of the subspaces obtained by applying appropriate weight-raising operators on the spaces of holomorphic modular forms and on the one-dimensional space spanned by the weight 2 nearly holomorphic Eisenstein series.

While Shimura's proof was classical, ours is representation-theoretic. We deduce the structure theorem from a decomposition for the space of \emph{$\mathfrak{n}$-finite} automorphic forms on $\SL_2(\R)$. To prove this decomposition, we use the mechanism of \emph{category $\mathcal{O}$} and a careful analysis of the various possible indecomposable submodules. It is possible to achieve the same end by more direct methods, but we prefer this approach as it generalizes to other groups.

This note may be viewed as the toy case of our paper \cite{PSS14}, where we prove an analogous structure theorem for vector-valued nearly holomorphic Siegel modular forms of degree two.
\end{abstract}

 \maketitle

\section{Nearly holomorphic functions}
Let $\HH_1$ be the complex upper half plane. Let $N^p(\HH_1)$ be the space of functions $f:\:\HH_1\to\C$ of the form
$$
 f(\tau)=\sum_{j=0}^pf_j(\tau)y^{-j},\qquad\tau=x+iy,
$$
where $f_0,\ldots,f_p$ are holomorphic functions on $\HH_1$. Any element of the space $N(\HH_1)=\cup_{p=0}^\infty N^p(\HH_1)$ is called a nearly holomorphic function on $\HH_1$.  It is an exercise to show that
\begin{equation}\label{holcoeffuniqueeq}
 \sum_{j=0}^pf_j(\tau)y^{-j}=0\qquad\Longleftrightarrow\qquad f_j=0\text{ for all }j=0,\ldots,p.
\end{equation}
Hence, the holomorphic coefficients of a nearly holomorphic function are unique\-ly determined.

If $f$ is a nearly holomorphic function, and if there exists a non-zero real number $r$ such that $f(\tau+r)=f(\tau)$ for all $\tau\in\HH_1$, then the holomorphic coefficients $f_j$ of $f$ exhibit the same translation invariance; this follows from \eqref{holcoeffuniqueeq}. Each $f_j$ therefore admits a Fourier expansion $f_j(\tau)=\sum a_j(n)e^{2\pi in\tau/r}$. It follows that $f$ admits a Fourier expansion whose coefficients are polynomials in $y^{-1}$.

For an integer $k$, we define the weight $k$ slash operator on functions $f:\:\HH_1\to\C$ in the usual way:
$$
 (f|_kg)(\tau)=(c\tau+d)^{-k}f\Big(\frac{a\tau+b}{c\tau+d}\Big),\qquad g=\mat{a}{b}{c}{d}\in\SL_2(\R).
$$
Let $\Gamma$ be a congruence subgroup of $\SL_2(\Q)$. Let $N_k^p(\Gamma)$ denote the space of functions $F:\H_1\rightarrow\C$ such that
\begin{enumerate}
 \item $F \in N^p(\H_1)$;
 \item $F|_k\gamma = F$ for all $\gamma \in \Gamma$;
 \item $F$ satisfies the cusp condition. (This notion is defined in terms of Fourier expansions just as in the case of holomorphic modular forms; see, e.g., \S2.1 of \cite{Miyake1989}.)
\end{enumerate}
We denote by $N_k^p(\Gamma)^\circ$ the subspace of functions that vanish at every cusp. The space
$N_k(\Gamma)=\cup_{p=0}^\infty N_k^p(\Gamma)$ is the space of \emph{nearly holomorphic modular forms} with respect to $\Gamma$, and $N_k(\Gamma)^\circ = \cup_{p=0}^\infty N_k^p(\Gamma)^\circ$ is the space of \emph{nearly holomorphic cusp forms}. Evidently, $M_k(\Gamma):=N_k^0(\Gamma)$ is the usual space of holomorphic modular forms of weight $k$ with respect to $\Gamma$, and $S_k(\Gamma):=N_k^0(\Gamma)^\circ$ is the subspace of cusp forms. Nearly holomorphic modular forms occur naturally as special values of Eisenstein series and thus their arithmetic properties imply arithmetic properties for various $L$-functions via the theory of Rankin-Selberg type integrals. We refer the reader to the introduction of~\cite{PSS14} for further remarks in this direction.

For an integer $k$, we define the classical Maass weight raising and lowering operators $R_k,L_k$ on the space of smooth functions on $\HH_1$ by
\begin{equation}\label{RkLkdefeq}
  R_k=\frac ky+2i\frac{\partial}{\partial\tau},\qquad L_k=-2iy^2\frac{\partial}{\partial\bar\tau},
\end{equation}
where $\frac{\partial}{\partial\tau}=\frac12(\frac{\partial}{\partial x}-i\frac{\partial}{\partial y})$ and $\frac{\partial}{\partial\bar\tau}=\frac12(\frac{\partial}{\partial x}+i\frac{\partial}{\partial y})$ are the usual Wirtinger derivatives. Also define an operator $\Omega_k$ by
\begin{equation}\label{OmegakRkLkreleq}
 \Omega_k=\frac14k^2+\frac12R_{k-2}L_k+\frac12L_{k+2}R_k.
\end{equation}
A calculation shows that
\begin{equation}\label{Omegakdefeq}
 \Omega_k=y^2\Big(\frac{\partial^2}{\partial x^2}+\frac{\partial^2}{\partial y^2}\Big)-2iky\frac{\partial}{\partial\bar\tau}+\frac k2\Big(\frac k2-1\Big).
\end{equation}
The following lemma is readily verified.
\begin{lemma}\label{RkLkNplemma}
 Let $k$ be an integer, and $p$ be a non-negative integer. Let $\Gamma$ be a congruence subgroup of $\SL_2(\Q)$.
 \begin{enumerate}
  \item $R_k$ induces maps $N^p_k(\Gamma)\to N^{p+1}_{k+2}(\Gamma)$ and $N^p_k(\Gamma)^\circ\to N^{p+1}_{k+2}(\Gamma)^\circ$.
  \item $L_k$ induces maps $N^p_k(\Gamma)\to N^{p-1}_{k-2}(\Gamma)$ and $N^p_k(\Gamma)^\circ\to N^{p-1}_{k-2}(\Gamma)^\circ$.
  \item $\Omega_k$ induces endomorphisms of $N^p_k(\Gamma)$ and of $N^p_k(\Gamma)^\circ$.
 \end{enumerate}
 Here, we understand $N^p_k(\Gamma)=N_k^p(\Gamma)^\circ=0$ for $p<0$.
\end{lemma}
Henceforth, we drop the subscripts and let $R$, $L$, and $\Omega$ denote the operators on $\bigoplus_{k}N_k(\Gamma)$ whose restrictions to $N_k(\Gamma)$ are given by $R_k$, $L_k$, and $\Omega_k$, respectively.
\begin{lemma}\label{Nkpfindimlemma}
 For any integer $k$ and non-negative integer $p$, the space $N^p_k(\Gamma)$ is finite-dimensional.
\end{lemma}
\begin{proof}
This is well known for $p=0$, since $N^0_k(\Gamma)=M_k(\Gamma)$ is simply the space of holomorphic modular forms of weight $k$. For $p>0$ there is an exact sequence
$$
 0\longrightarrow M_k(\Gamma)\longrightarrow N_k^p(\Gamma)\stackrel{L}{\longrightarrow}N_{k-2}^{p-1}(\Gamma).
$$
Hence the assertion follows by induction on $p$.
\end{proof}

The following well-known fact will be important for our arguments further below. (For a proof, see Theorem 2.5.2 of \cite{Miyake1989}.)
\begin{lemma}\label{holkneglemma}
 $S_k(\Gamma)=0$ if $k\leq0$, and $M_k(\Gamma)=0$ if $k<0$. The space $M_0(\Gamma)$ consists of the constant functions.
\end{lemma}

\section{Representations of \texorpdfstring{$\SL_2(\R)$}{} and differential operators}
To reinterpret elements of $N^p_k(\Gamma)$ as functions on $\SL_2(\R)$, we recall the basic representation theory of this group. Let $\mathfrak{g}=\mathfrak{sl}_2(\R)$ be the Lie algebra of $\SL_2(\R)$, consisting of all $2\times2$ real matrices with trace zero. Let $\mathfrak{g}_\C=\mathfrak{sl}_2(\C)$ be its complexification. The elements
\begin{equation}\label{HRLdefeq}
 H=-i\mat{0}{1}{-1}{0},\qquad R=\frac12\mat{1}{i}{i}{-1},\qquad L=\frac12\mat{1}{-i}{-i}{-1}
\end{equation}
of $\mathfrak{g}_\C$ satisfy the relations $[H,R]=2R$, $[H,L]=-2L$ and $[R,L]=H$. The \emph{Casimir element} is the element in the universal enveloping algebra $\mathcal{U}(\mathfrak{g}_\C)$ given by
\begin{equation}\label{casimirdefeq}
 \Omega=\frac14 H^2+\frac12RL+\frac12LR.
\end{equation}
Then $\Omega$ lies in the center $\mathcal{Z}$ of $\mathcal{U}(\mathfrak{g}_\C)$, and it is known that $\mathcal{Z}=\C[\Omega]$.

Let $K=\SO(2)$ be the standard maximal compact subgroup of $\SL_2(\R)$, consisting of all elements $r(\theta)=\mat{\cos(\theta)}{\sin(\theta)}{-\sin(\theta)}{\cos(\theta)}$ with $\theta\in\R$.
By ``representation of $\SL_2(\R)$'' we mean a $(\mathfrak{g},K)$-module. In such a module $(\pi,V)$, we say a non-zero $v\in V$ has weight $k$ if
$$
 \pi(r(\theta))v=e^{ik\theta}v\qquad\text{for }\theta\in\R,
$$
or equivalently, $\pi(H)v=kv$. In an irreducible representation, all weights have the same parity, and every weight occurs at most once. The operator $\pi(R)$ raises the weight by $2$, and the operator $\pi(L)$ lowers the weight by $2$. The \emph{weight structure} of an irreducible representation is the list of weights, written in order. The following is the complete list of irreducible, admissible $(\mathfrak{g},K)$-modules.
\begin{enumerate}
 \item \emph{Finite-dimensional representations}. For a positive integer $p$, let $\mathcal{F}_p$ be the irreducible finite-dimensional representation of $\SL_2(\R)$ with weight structure $[-p+1,-p+3,\ldots,p-3,p-1]$. Hence $\dim\mathcal{F}_p=p$.
 \item \emph{Discrete series representations}. For a positive integer $p$ we denote by $\mathcal{D}_{p,+}$ the discrete series representation of $\SL_2(\R)$ with weight structure $[p+1,\:p+3,\:\ldots]$. Similarly, let $\mathcal{D}_{p,-}$ be the discrete series representation of $\SL_2(\R)$ with weight structure $[\ldots,\:-p-3,\:-p-1]$. Hence, $p$ is not the minimal weight of $\mathcal{D}_{p,+}$, but the Harish-Chandra parameter.
 \item \emph{Limits of discrete series}. Let $\mathcal{D}_{0,+}$ be the irreducible representation of $\SL_2(\R)$ with weight structure $[1,\,3,\,5,\ldots]$, and let $\mathcal{D}_{0,-}$ be the irreducible representation of $\SL_2(\R)$ with weight structure $[\ldots,-5,\,-3,\,-1]$. Formally these representations look like members of the discrete series, but they are not square-integrable.
 \item \emph{Principal series representations}. Their weight structure is either $2\Z$ or $2\Z+1$. For our purposes, all we need to know about principal series representations is that the operators $R$ and $L$ act injectively on such a $(\mathfrak{g},K)$-module.
\end{enumerate}
\subsection*{Functions on \texorpdfstring{$\SL_2(\R)$}{} and functions on \texorpdfstring{$\HH_1$}{}}
Let $W(k)$ be the space of smooth functions $\Phi:\:\SL_2(\R)\to\C$ with the property $\Phi(gr(\theta))=e^{ik\theta}\Phi(g)$ for all $\theta\in\R$ and $g\in\SL_2(\R)$. These are the vectors of weight $k$ under the right translation action on the space of smooth functions. The operator $R$ induces a map $W(k)\to W(k+2)$, and $L$ induces a map $W(k)\to W(k-2)$. Let $W$ be the space of smooth functions on $\HH_1$. For $\Phi\in W(k)$ we define an element $\tilde\Phi\in W$ by
\begin{equation}\label{PhitildePhieq}
 \tilde\Phi(x+iy)=y^{-k/2}\,\Phi(\mat{1}{x}{}{1}\mat{y^{1/2}}{}{}{y^{-1/2}}).
\end{equation}
It is straightforward to verify that
\begin{equation}\label{PhitildePhieq2}
 (\tilde\Phi\big|_kg)(i)=\Phi(g)\qquad\text{for all }g\in\SL_2(\R).
\end{equation}
The map $\Phi\mapsto\tilde\Phi$ establishes an isomorphism $W(k)\cong W$.
\begin{lemma}\label{diffopHlemma}
 Let $R ,L,\Omega$ be the operators on $W$ defined in \eqref{RkLkdefeq} and \eqref{Omegakdefeq}. Then the diagrams
 $$
  \begin{CD}
   W(k)@>\sim>>W&\qquad\qquad&&W(k)@>\sim>>W&\qquad\qquad&&W(k)@>\sim>>W\\
   @V{L}VV @VV{L}V&@V{R}VV @VV{R}V&@V{\Omega}VV @VV{\Omega}V\\
   W(k-2)@>\sim>>W&&&W(k+2)@>\sim>>W&&&W(k)@>\sim>>W
  \end{CD}
 $$
 are commutative.
\end{lemma}
\begin{proof}
The assertions for $R$ and $L$ follow from straightforward calculations. The assertion for $\Omega$ then follows from \eqref{OmegakRkLkreleq} and \eqref{casimirdefeq}.
\end{proof}
The previous lemma is about smooth functions only and does not involve any transformation properties. If $\Phi\in W(k)$ satisfies $\Phi(\gamma g)=\Phi(g)$ for all $g\in\SL_2(\R)$ and all elements $\gamma$ of a congruence subgroup $\Gamma$, then it follows from \eqref{PhitildePhieq2} that $\tilde\Phi|_k\gamma=\tilde\Phi$ for all $\gamma\in\Gamma$. Conversely, given a smooth function $f$ on $\HH_1$ satisfying $f|_k\gamma=f$ for all $\gamma\in\Gamma$, we may consider the function $\Phi\in W(k)$ such that $\tilde\Phi=f$. This function is then left $\Gamma$-invariant. We will see in the next subsection that if $f\in N^p_k(\Gamma)$, then $\Phi$ is an automorphic form.

\section{The structure theorem for cusp forms}

Let $\Gamma\subset\SL_2(\Q)$ be a congruence subgroup. Let $\AA(\Gamma)$ be the space of automorphic forms on $\SL_2(\R)$, and let $\AA(\Gamma)^\circ$ be the subspace of cusp forms. Recall that automorphic forms are required to be smooth, left $\Gamma$-invariant, $K$-finite, $\mathcal{Z}$-finite, and slowly increasing; we refer to \cite{BorelJacquet1979} for the precise definitions. The spaces $\AA(\Gamma)$ and $\AA(\Gamma)^\circ$ are $(\mathfrak{g},K)$-modules with respect to right translation. Let $\AA_k(\Gamma)$ (resp.\ $\AA_k(\Gamma)^\circ$) be the space of automorphic forms (resp.\ cusp forms) $\Phi$ satisfying $H.\Phi=k\Phi$, or equivalently, $\Phi(gr(\theta))=e^{ik\theta}\Phi(g)$ for all $\theta\in\R$ and $g\in\SL_2(\R)$.

If $f\in\AA(\Gamma)$ and $g\in\AA(\Gamma)^\circ$, then the function $|fg|$ is integrable over $\Gamma\backslash\SL_2(\R)$. In particular, $\AA(\Gamma)^\circ\subset L^2(\Gamma\backslash\SL_2(\R))$. With respect to the $L^2$ inner product, the space $\AA(\Gamma)^\circ$ decomposes into an orthogonal direct sum of irreducible representations, each occurring with finite multiplicity.

Let $\Phi\in\AA(\Gamma)$. We will say that $\Phi$ is \emph{$\mathfrak{n}$-finite} if $L^v\Phi=0$ for large enough $v$. Let $\AA(\Gamma)_{\mathfrak{n}\text{-fin}}$ be the space of $\mathfrak{n}$-finite automorphic forms, and let $\AA(\Gamma)^\circ_{\mathfrak{n}\text{-fin}}$ be the subspace of $\mathfrak{n}$-finite cusp forms. The following properties are easy to verify:
\begin{itemize}
 \item $\AA(\Gamma)_{\mathfrak{n}\text{-fin}}$ is a $(\mathfrak{g},K)$-submodule of $\AA(\Gamma)$.
 \item $\AA(\Gamma)_{\mathfrak{n}\text{-fin}}$ is the direct sum of its weight spaces, i.e.: If $\Phi\in\AA(\Gamma)_{\mathfrak{n}\text{-fin}}$ and $\Phi=\Phi_1+\ldots+\Phi_m$ with $\Phi_i\in\AA_{k_i}(\Gamma)$ for different weights $k_i$, then $\Phi_i\in\AA(\Gamma)_{\mathfrak{n}\text{-fin}}$ for each $i$.
\end{itemize}
Analogous statements hold for cusp forms.
\begin{lemma}\label{Npkautformlemma}
 Let $k$ be an integer, and $p$ a non-negative integer. Let $\Gamma$ be a congruence subgroup of $\SL_2(\Q)$. Let $f\in N_k^p(\Gamma)$ be non-zero. Define a function $\Phi$ on $\SL_2(\R)$ by $\Phi(g)=(f|_kg)(i)$. Then $\Phi\in\AA_k(\Gamma)_{\mathfrak{n}\text{-fin}}$. If $f$ is a cusp form, then $\Phi\in\AA_k(\Gamma)^\circ_{\mathfrak{n}\text{-fin}}$.
\end{lemma}
\begin{proof}
Evidently, $\Phi$ is smooth, left $\Gamma$-invariant and has weight $k$. Since $N^p_k(\Gamma)$ is finite-dimensional (Lemma \ref{Nkpfindimlemma}) and $\Omega$ acts on $N^p_k(\Gamma)$ (Lemma \ref{RkLkNplemma}), the function $f$ is $\C[\Omega]$-finite. Hence, by Lemma \ref{diffopHlemma}, the function $\Phi$ is $\mathcal{Z}$-finite. The holomorphy of $f$ at the cusps implies that $\Phi$ is slowly increasing. This proves $\Phi\in\AA_k(\Gamma)$. Cuspidality of $f$ translates into cuspidality of $\Phi$. To prove $\mathfrak{n}$-finiteness, observe that $L^vf=0$ for large enough $v$ by Lemma \ref{RkLkNplemma}. Hence $L^v\Phi=0$ for large enough $v$ by Lemma \ref{diffopHlemma} and Lemma~\ref{holkneglemma}.
\end{proof}

The following result is sometimes called the ``duality theorem''; see Theorem 2.10 of \cite{Gelbart1975}.
\begin{proposition}\label{AA0nfindecompprop}
 As $(\mathfrak{g},K)$-modules, we have
 $$
  \AA(\Gamma)^\circ_{\mathfrak{n}\text{-fin}}=\bigoplus_{\ell=1}^\infty n_\ell\mathcal{D}_{\ell-1,+},\qquad n_\ell=\dim S_\ell(\Gamma).
 $$
 The lowest weight vectors in the isotypical component $n_\ell\mathcal{D}_{\ell-1,+}$ correspond to elements of $S_\ell(\Gamma)$ via the map $\Phi\mapsto\tilde\Phi$, where $(\tilde\Phi|_kg)(i)=\Phi(g)$ for $g\in\SL_2(\R)$.
\end{proposition}
\begin{proof}
Since $\AA(\Gamma)^\circ_{\mathfrak{n}\text{-fin}}$ is a $(\mathfrak{g},K)$-submodule of $\AA(\Gamma)^\circ$, it decomposes into an orthogonal direct sum of irreducible $(\mathfrak{g},K)$-modules. None of the irreducible constituents can be of the form $\mathcal{D}_{p,-}$ or a principal series representation, since any non-zero vector in such a constituent would not be $\mathfrak{n}$-finite. Neither can $\AA(\Gamma)^\circ_{\mathfrak{n}\text{-fin}}$ contain any finite-dimensional representations; the lowest weight vector in such a constituent would give rise to a holomorphic cusp form of non-positive weight, which is not possible by Lemma \ref{holkneglemma}. It follows that $\AA(\Gamma)^\circ_{\mathfrak{n}\text{-fin}}$ can only contain constituents of the form $\mathcal{D}_{\ell-1,+}$ for $\ell\geq1$. The fact that $\mathcal{D}_{\ell-1,+}$ occurs with multiplicity $\dim S_\ell(\Gamma)$ follows because a lowest weight vector in a constituent of the form $\mathcal{D}_{\ell-1,+}$ gives rise to an element of
$S_\ell(\Gamma)$, and conversely.
\end{proof}
\begin{remark}
 It follows from Proposition \ref{AA0nfindecompprop} that $\AA(\Gamma)^\circ_{\mathfrak{n}\text{-fin}}$ is an admissible $(\mathfrak{g},K)$-module.
\end{remark}
Knowing Proposition \ref{AA0nfindecompprop}, it is now easy to derive the following Structure Theorem for cuspidal nearly holomorphic modular forms.
\begin{theorem}[Structure theorem for cusp forms]\label{GL2cuspstructuretheorem}
 Fix non-nega\-tive integers $k,p$ and a congruence subgroup $\Gamma$ of $\SL_2(\Q)$. There is an orthogonal direct sum decomposition
 \begin{equation}\label{GL2cuspstructuretheoremeq1}
    N_k^p(\Gamma)^\circ =  \bigoplus_{\substack{\ell\ge 1\\\ell\equiv k\bmod{2}\\k-2p\le\ell\le k}} R^{(k-\ell)/2}\left(S_\ell(\Gamma)\right).
 \end{equation}
 In particular, $N_0^p(\Gamma)^\circ=0$ and $N_1^p(\Gamma)^\circ=S_1(\Gamma)$.
\end{theorem}
\begin{proof}
Let $f\in N_k^p(\Gamma)$. Define a function $\Phi$ on $\SL_2(\R)$ by $\Phi(g)=(f|_kg)(i)$. Then $\Phi\in\AA_k(\Gamma)^\circ_{\mathfrak{n}\text{-fin}}$ by Lemma \ref{Npkautformlemma}. If $f$ has weight $0$, then $f=0$, since the weight $0$ does not occur in $\AA_k(\Gamma)^\circ_{\mathfrak{n}\text{-fin}}$ by Proposition \ref{AA0nfindecompprop}. Assume in the following that $k\geq1$ and that $f$ is non-zero.

Write $\Phi=\sum\Phi_i$, where each $\Phi_i\in\AA_k(\Gamma)^\circ_{\mathfrak{n}\text{-fin}}$ generates an irreducible $(\mathfrak{g},K)$-module $V_i\cong\mathcal{D}_{\ell_i-1,+}$ with $\ell_i\geq1$; this is possible by Proposition \ref{AA0nfindecompprop}. Evidently, $f=\sum\tilde\Phi_i$, where $\tilde\Phi_i$ is the function on $\HH_1$ corresponding to $\Phi_i$ via \eqref{PhitildePhieq}.

Since $f\in N_k^p(\Gamma)^\circ$, it follows from (2) of Lemma \ref{RkLkNplemma} that $L^{p+1}f=0$, thus $L^{p+1}\Phi=0$ by Lemma \ref{diffopHlemma}, and then also $L^{p+1}\Phi_i=0$ for all $i$. The weight of $L^{p+1}\Phi_i$ being $k-2p-2$, it follows that $V_i$ only contains weights greater or equal to $k-2p$. Hence $\ell_i\geq k-2p$ for all $i$.

Let $\Phi_{i,0}\in V_i$ be a lowest weight vector; thus $\Phi_{i,0}$ has weight $\ell_i\leq k$, and $\ell_i\equiv k$ mod $2$. The corresponding function $\tilde\Phi_{i,0}$ on $\HH_1$ is an element of $S_{\ell_i}(\Gamma)$. Since every weight occurs only once in $V_i$, we have $R^{(k-\ell_i)/2}\Phi_{i,0}=c_i\Phi_i$ for some non-zero constant $c_i$. By Lemma \ref{diffopHlemma}, it follows that
$$
R^{(k-\ell_i)/2}\tilde\Phi_{i,0}=c_i\tilde\Phi_i,
$$
and hence
$$
 f=\sum\tilde\Phi_i=\sum c_i^{-1}R^{(k-\ell_i)/2}\tilde\Phi_{i,0}\in\sum_{\substack{\ell\ge 1\\\ell\equiv k\bmod{2}\\k-2p\le\ell\le k}} R^{(k-\ell)/2}\left(S_\ell(\Gamma)\right).
$$
This proves that the left hand side of \eqref{GL2cuspstructuretheoremeq1} is contained in the right hand side. The orthogonality of the right hand side follows from the above construction and the fact that the isotypical components in Proposition \ref{AA0nfindecompprop} are orthogonal; observe that the map $\Phi\mapsto\tilde\Phi$ is isometric with respect to the $L^2$-scalar product on the left hand side and the Petersson inner product on the right hand side.
\end{proof}
\begin{remark}
It is well known, or follows from an easy calculation, that $\Omega$ acts on $\mathcal{D}_{\ell-1,+}$ by the scalar $\frac12\ell(\frac12\ell-1)$. Hence, by Lemma \ref{diffopHlemma}, $\Omega$ acts on the subspace $R^{(k-\ell)/2}\left(S_\ell(\Gamma)\right)$ of $N_k^p(\Gamma)^\circ$ by $\frac12\ell(\frac12\ell-1))$. In particular, $\Omega$ acts diagonalizably on $N_k^p(\Gamma)^\circ$, and the pieces in the decomposition \eqref{GL2cuspstructuretheoremeq1} can be intrinsically characterized as the eigenspaces with respect to $\Omega$.
\end{remark}
\subsection*{Petersson inner products}
For $f,g \in N_k(\Gamma)$ with at least one of them in $N_k(\Gamma)^\circ$, we define the Petersson inner product $\langle f, g \rangle$ by the equation
$$
 \langle f, g \rangle = \vl(\Gamma \bs \H_1)^{-1}\int\limits_{\Gamma \bs \H_1} f(\tau) \overline{g(\tau)}\,\frac{dx dy}{y^2}.
$$
It can be easily checked that
\begin{equation}\label{peterssonequal}
 \langle f, g \rangle = \langle \Phi_f, \Phi_g \rangle,
\end{equation}
where $\Phi_f(h)=(f|_kh)(i)$ (and $\Phi_g$ is defined similarly) and the inner product of  $\Phi_f$ and $\Phi_g$ is defined by
$$
 \langle \Phi_f, \Phi_g \rangle = \frac1{\vl(\SL_2(\Z) \bs \SL_2(\R))}\int\limits_{\SL_2(\Z) \bs \SL_2(\R)} \Phi_f(h) \overline{\Phi_g(h)}\,dh.
$$
We define the subspace $\E_k(\Gamma)$ to be the orthogonal complement of $N_k(\Gamma)^\circ$ in $N_k(\Gamma)$. Let $\E_k^p(\Gamma)=\E_k(\Gamma)\cap N_k^p(\Gamma)$. We write $E_k(\Gamma)$ to mean $\E_k^0(\Gamma)$. In Corollary \ref{Nkporthogonalcorollary} below we will prove that $N_k^p(\Gamma)$ is the orthogonal direct sum of $N_k^p(\Gamma)^\circ$ and $\E_k^p(\Gamma)$.

\begin{lemma}\label{slashEorthogonallemma}
 Let $k$ be a non-negative integer. Let $f\in\E_k(\Gamma)$, and let $\Phi_f\in\AA(\Gamma)_{\mathfrak{n}\text{-fin}}$ be the corresponding function on $\SL_2(\R)$. Then $\Phi_f$ is orthogonal to $\AA(\Gamma)_{\mathfrak{n}\text{-fin}}^\circ$.
\end{lemma}
\begin{proof}
Let $\Psi\in\AA(\Gamma)_{\mathfrak{n}\text{-fin}}^\circ$; we have to show that $\langle\Phi_f,\Psi\rangle=0$. We may assume that $\Psi$ generates an irreducible module $\mathcal{D}_{\ell-1,+}$ for some $\ell\geq1$. Since $\Phi_f$ has weight $k$, we may assume that $\Psi$ does as well. But then $\Psi$ corresponds to an element $g$ of $N_k(\Gamma)^\circ$. By hypothesis $\langle f,g\rangle=0$. Hence $\langle\Phi_f,\Psi\rangle=0$ by \eqref{peterssonequal}.
\end{proof}

\begin{lemma}\label{preserveeis}
 Let $k \ge 1$ and $v\ge 0$ be integers. Then $R^v$ takes $S_k(\Gamma)$ to $N_{k+2v}^{v}(\Gamma)^\circ$ and $E_k(\Gamma)$ to $\E_{k+2v}^{v}(\Gamma)$.
\end{lemma}
\begin{proof} The fact that $R^v$ takes $S_k(\Gamma)$ to $N_{k+2v}^{v}(\Gamma)^\circ$ is an immediate consequence of the fact that the differential operator $R$ commutes with the $|_k$ operator and does not increase the support of the Fourier coefficients.

Let us show that $R^v$ takes $E_k(\Gamma)$ to $\E_{k+2v}^{ v}(\Gamma)$. Let $f \in E_k(\Gamma)$. In view of~\eqref{peterssonequal}, it suffices to show that $R^v(\Phi_f)$ and $\Phi_g$ are orthogonal for all $g \in  N_{k+2v}^{v}(\Gamma)^\circ$. But note that $\U(\mathfrak{g}_\C)\Phi_f$ and $\U(\mathfrak{g}_\C)\Phi_g$ are orthogonal submodules of $\AA(\Gamma)$ (as $\U(\mathfrak{g}_\C)\Phi_g$ is completely contained in $\AA(\Gamma)^\circ$ and $\U(\mathfrak{g}_\C)\Phi_f$ is contained in the orthogonal complement of $\AA(\Gamma)^\circ$ by Lemma \ref{slashEorthogonallemma}). Hence $R^v(\Phi_f)$ and $\Phi_g$ are orthogonal.
\end{proof}
\begin{lemma}\label{peterssoninvcusp}Let $f \in S_k(\Gamma)$. Then for all $v \ge 0$, there exists a constant $c_{k,v}$ (depending only on $k, v$) such that $$\langle R^v(f), R^v(f) \rangle  = c_{k,v}\langle f , f\rangle.$$
\end{lemma}
\begin{proof}
Consider the $(\mathfrak{g},K)$ module $\mathcal{D}_{k-1,+}$ and let $v_0$ be a lowest-weight vector in it. Note that $v_0$ is unique up to multiples. It is well-known that $\mathcal{D}_{k-1,+}$ is unitarizable; let $\langle, \rangle$ denote the (unique up to multiples) $\mathfrak{g}$-invariant inner product on it. Put $c_{k,v} = \langle R^v(v_0), R^v(v_0) \rangle /\langle v_0 , v_0\rangle.$ Note that $c_{k,v}$ does not depend on the choice of model for $\mathcal{D}_{k-1,+}$, the choice of $v_0$ or the normalization of inner products.

Now all we need to observe is that the automorphic form $\Phi_f\in\AA(\Gamma)_{\mathfrak{n}\text{-fin}}^\circ$ corresponding to $f$ generates a module isomorphic to $\mathcal{D}_{k-1,+}$, that $\Phi_f$ is a lowest weight vector in this module, and \eqref{peterssonequal}.
\end{proof}

\begin{proposition}\label{peterssoninvgeneral}Let $f \in M_k(\Gamma)$, $g \in S_k(\Gamma)$. Then $$\langle R^v(f), R^v(g) \rangle  = c_{k,v}\langle f , g\rangle,$$ where the constant $c_{k,v}$ is as in the previous lemma.
\end{proposition}
\begin{proof}Because of Lemma~\ref{preserveeis}, we may assume that $f$ and $g$ both belong to $S_k(\Gamma)$. Now the Proposition follows by applying the previous lemma to $f+g$.
\end{proof}
\section{The non-cuspidal case}
\subsection*{The obstruction in the non-cuspidal case}
The Structure Theorem \ref{GL2cuspstructuretheorem} cannot hold without modifications for non-cuspidal nearly holomorphic modular forms. The reason is the existence of the weight $2$ Eisenstein series
\begin{equation}\label{e2defeq}
 E_2(\tau)=-\frac{3}{\pi y}+1-24\sum_{n=1}^\infty\sigma_1(n)e^{2\pi in\tau},\qquad\sigma_1(n)=\sum_{d|n}d.
\end{equation}
As is well known, $E_2$ is modular with respect to $\SL_2(\Z)$; thus, $E_2\in N^1_2(\Gamma)$ for any congruence subgroup $\Gamma$ of $\SL_2(\Z)$. But evidently $E_2$ cannot be obtained via raising operators from holomorphic forms of lower weight, since the only modular forms of weight $0$ are the constant functions.

Let $\Phi_2\in\AA_2(\Gamma)_{\mathfrak{n}\text{-fin}}$ be the automorphic form corresponding to $E_2$ via Lemma \ref{Npkautformlemma}. Let $V_{\Phi_2}$ be the $(\mathfrak{g},K)$-module generated by $\Phi_2$. Since $L_2\Phi_2=\frac3\pi$, we have $L\Phi_2\in\C$ by Lemma \ref{diffopHlemma}. The weight structure of $V_{\Phi_2}$ is therefore $[0,2,4,\ldots]$, and the constant functions are a submodule of $V_{\Phi_2}$. More precisely, there is an exact sequence
\begin{equation}\label{E2exactseqeq}
 0\longrightarrow\C\longrightarrow V_{\Phi_2}\longrightarrow\mathcal{D}_{1,+}\longrightarrow0;
\end{equation}
recall that $\mathcal{D}_{1,+}$ is the lowest weight module with weight structure $[2,4,6,\ldots]$. Clearly, this sequence does not split. Consequently, unlike in the cuspidal case, the space $\AA(\Gamma)_{\mathfrak{n}\text{-fin}}$ is not the direct of irreducible submodules. However, the following result states that $V_{\Phi_2}$ represents the only obstruction:
\begin{proposition}\label{AAnfindecompprop}
 As $(\mathfrak{g},K)$-modules, we have
 $$
  \AA(\Gamma)_{\mathfrak{n}\text{-fin}}=V_{\Phi_2}\:\oplus\;\bigoplus_{\ell=1}^\infty n_\ell\mathcal{D}_{\ell-1,+},\qquad n_\ell=\dim M_\ell(\Gamma).
 $$
 The lowest weight vectors in the isotypical component $n_\ell\mathcal{D}_{\ell-1,+}$ correspond to elements of $M_\ell(\Gamma)$ via the map $\Phi\mapsto\tilde\Phi$, where $(\tilde\Phi|_kg)(i)=\Phi(g)$ for $g\in\SL_2(\R)$. The module $V_{\Phi_2}$ sits in the exact sequence \eqref{E2exactseqeq} and is generated by the function $\Phi_2$ such that $\tilde\Phi_2=E_2$.
\end{proposition}
To prove this result, we will set up a certain algebraic apparatus. It turns out that the mechanism of \emph{category $\mathcal{O}$} is well suited toward our problem. In the $\SL_2$ case this mechanism could be replaced by more direct arguments, but we prefer to use category $\mathcal{O}$ because this method generalizes to the $\Sp_4$ case; see \cite{PSS14}. Our reference for category $\mathcal{O}$ will be \cite{Humphreys2008}.
\subsection*{Roots and weights}
Let $H$, $R$, $L$ be the elements of $\mathfrak{g}_\C$ defined in \eqref{HRLdefeq}. Then $\mathfrak{h}=\langle H\rangle$ is a Cartan subalgebra of $\mathfrak{g}_\C$. Let $\Phi\subset\mathfrak{h}^*$ be the root system of $\mathfrak{g}_\C$ corresponding to $\mathfrak{h}$. Then $R$ and $L$ span the corresponding root spaces. We identify an element $\lambda$ of $\mathfrak{h}^*$ with the element $\lambda(H)$ of $\C$. Then $\Phi=\{\pm2\}$. Let $E$ be the $\R$-span of $\Phi$. We endow $E$ with the inner product $(\cdot,\cdot)$ given by the usual multiplication of real numbers. Perhaps counterintuitively, we will declare $-2$ to be a positive (and simple) root, with corresponding root vector $L$, and $+2$ a negative root, with corresponding root vector $R$. The \emph{weight lattice} $\Lambda$ is defined as
\begin{equation}\label{weightlatticeeq}
 \Lambda=\Big\{\lambda\in E\:\Big|\:2\frac{(\lambda,\alpha)}{(\alpha,\alpha)}\in\Z\;\text{for all }\alpha\in\Phi\Big\}.
\end{equation}
Evidently, $\Lambda=\Z\subset E$. There is an ordering on $\Lambda$ defined as follows:
\begin{equation}\label{dominanceordereq}
 \mu\preccurlyeq\lambda\quad\Longleftrightarrow\quad\lambda\in\mu+\Gamma,
\end{equation}
where $\Gamma$ is the set of all $\Z_{\geq0}$-multiples of the positive root. Hence,
$$
 \mu\preccurlyeq\lambda\quad\Longleftrightarrow\quad\lambda\leq\mu\text{ and }\lambda\equiv\mu\bmod2.
$$
The \emph{fundamental weight} is $-1$, and the \emph{dominant integral weights} are the $\Z_{\geq0}$-multiples of the fundamental weight. We write $\Lambda^+$ for the set of dominant integral weights. Hence, $\Lambda^+=\{-1,-2,\ldots\}$. We write $\varrho=-1$ for half the sum of the positive roots. As before, let $\mathcal{Z}$ denote the center of the universal enveloping algebra $\mathcal{U}(\mathfrak{g}_\C)$. Via the Harish-Chandra isomorphism, all possible characters of $\mathcal{Z}$ are indexed by elements of $\mathfrak{h}^*$ modulo Weyl group action; see equation (8.32) in \cite{Knapp1986}. We denote by $\chi_{\lambda}$ the character of $\mathcal{Z}$ corresponding to $\lambda\in\mathfrak{h}^*$. Note that $\chi_{w\lambda}=\chi_{\lambda}$ for all $w\in W$, where the Weyl group acts by negation on $\mathfrak{h}^*\cong\C$.
\subsection*{Verma modules}
We recall the definition of the standard Verma modules. Let $\lambda$ be an integer, considered as an element of the weight lattice $\Lambda$. Let $\C_\lambda$ be the one-dimensional space on which $\mathfrak{h}=\langle H\rangle$ acts via $\lambda$. Let
$$
 \mathfrak{b}=\mathfrak{h}+\langle L\rangle
$$
be the Borel algebra defined by our positive system. We consider $\C_\lambda$ a $\mathfrak{b}$-module with the action of $L$ being trivial. Then the \emph{Verma module} corresponding to $\lambda$ is defined as
\begin{equation}\label{vermadefeq}
 N(\lambda)=\mathcal{U}(\mathfrak{g}_\C)\otimes_{\mathcal{U}(\mathfrak{b})}\C_\lambda.
\end{equation}
Clearly, $N(\lambda)$ contains the weight $\lambda$ with multiplicity one. Any non-zero vector $v$ in $N(\lambda)$ of weight $\lambda$ is called a \emph{highest weight vector}. It is well known that $N(\lambda)$ has the following properties:
\begin{enumerate}
 \item $N(\lambda)$ is a free module of rank $1$ over $\mathcal{U}(R)=\C[R]$.
 \item The set of weights of $N(\lambda)$ is $\lambda-\Gamma=\{\lambda,\lambda+2,\ldots\}$. Each weight occurs with multiplicity one.
 \item The module $N(\lambda)$ is a universal highest weight module for the weight $\lambda$, meaning it satisfies this universal property: Let $M$ be a $\mathfrak{g}_\C$-module which contains a vector $v$ with the following properties:
  \begin{itemize}
   \item $M=\mathcal{U}(\mathfrak{g}_\C)v$;
   \item $v$ has weight $\lambda$;
   \item $Lv=0$.
  \end{itemize}
  Then there exists a surjection $N(\lambda)\to M$ mapping a highest weight vector in $N(\lambda)$ to $v$.
 \item $N(\lambda)$ admits a unique irreducible submodule, and a unique irreducible quotient $L(\lambda)$. In particular, $N(\lambda)$ is indecomposable. See Theorem 1.2 of \cite{Humphreys2008}.
 \item $N(\lambda)$ has finite length. Each factor in a composition series is of the form $L(\mu)$ for some $\mu\leq\lambda$.
 \item $N(\lambda)$ admits the central character $\chi_{\lambda+\varrho}$, i.e., $\mathcal{Z}$ acts on $N(\lambda)$ via $\chi_{\lambda+\varrho}$. See Sects.\ 1.7--1.10 of \cite{Humphreys2008}. Note that Humphrey's $\chi_\lambda$ is our $\chi_{\lambda+\varrho}$.
 \item $L(\lambda)$ is finite-dimensional if and only if $\lambda\in\Lambda^+=\{0,-1,-2,\ldots\}$. See Theorem 1.6 of \cite{Humphreys2008}.
 \item $N(\lambda)$ is simple if and only if $\lambda>0$. See Theorem 4.4 of \cite{Humphreys2008}.
\end{enumerate}
Evidently, $L(0)$ is the trivial $\mathfrak{g}_\C$-module. It has central character $\chi_\varrho$.
\subsection*{Category \texorpdfstring{$\mathcal{O}$}{}}
We recall from Sect.~1.1 of \cite{Humphreys2008} the definition of category $\mathcal{O}$. This category is defined with respect to a choice of Cartan subalgebra $\mathfrak{h}$ and a choice of simple roots, and we make the choices specified above. Let $\mathfrak{n}$ be the space spanned by the positive root vectors, hence, in our case, $\mathfrak{n}=\langle L\rangle$. A $\mathfrak{g}_\C$-module $M$ is said to be in category $\mathcal{O}$ if it satisfies the following conditions:
\begin{enumerate}
 \item[($\mathcal{O}$1)] $M$ is a finitely generated $\mathcal{U}(\mathfrak{g}_\C)$-module.
 \item[($\mathcal{O}$2)] $M$ is the direct sum of its weight spaces, and all weights are integral.
 \item[($\mathcal{O}$3)] $M$ is locally $\mathfrak{n}$-finite. This means: For each $v\in M$ the subspace $\mathcal{U}(\mathfrak{n})v$ is finite-dimensional.
\end{enumerate}
Note that we are slightly varying the definition of category $\mathcal{O}$ by requiring that all weights are integral; the relevant results in \cite{Humphreys2008} still hold with this modification.

$\mathcal{O}$ is an abelian category. Evidently, $\mathcal{O}$ contains all Verma modules $N(\lambda)$ and their irreducible quotients $L(\lambda)$. The modules $M$ in $\mathcal{O}$ have many nice properties, as explained in the first sections of \cite{Humphreys2008}. For example:
\begin{itemize}
 \item $M$ has finite length, and admits a filtration
  \begin{equation}\label{categoryOeq1}
   0=V_0\subset V_1\subset\ldots\subset V_n\subset M,
  \end{equation}
   with $V_i/V_{i-1}\cong L(\lambda)$ for some $\lambda\in\mathfrak{h}^*$.
 \item $M$ can be written as a finite direct sum of indecomposable modules.
 \item If $M$ is an indecomposable module, then there exists a character $\chi$ of $\mathcal{Z}$ such that $M=N(\chi)$. Here,
  \begin{equation}\label{Mchidefeq}
   N(\chi)=\{v\in M\:|\:(z-\chi(z))^nv=0\text{ for some $n$ depending on $z$}\}.
  \end{equation}

\end{itemize}
For each $M$ in category $\mathcal{O}$ we may write
\begin{equation}\label{categoryOeq2}
 M=\bigoplus_\chi N(\chi),
\end{equation}
where $\chi$ runs over characters of $\mathcal{Z}$, and $N(\chi)$ is defined as in \eqref{Mchidefeq}; see Sect.~1.12 of \cite{Humphreys2008}. The \emph{modules} $N(\chi)$ may or may not be indecomposable.

Another feature of $\mathcal{O}$ is the existence of a duality functor $M\mapsto M^\vee$, as explained in Sect.~3.2 of \cite{Humphreys2008}. In general, $M^\vee$ is \emph{not} the contragredient of $M$, as $\mathcal{O}$ is not closed under taking contragredients. The duality functor in $\mathcal{O}$ has the following properties:
\begin{itemize}
 \item $M\mapsto M^\vee$ is exact and contravariant.
 \item $M^{\vee\vee}\cong M$.
 \item $(M_\chi)^\vee\cong(M^\vee)_\chi$ for a character $\chi$ of $\mathcal{Z}$.
 \item $L(\lambda)^\vee\cong L(\lambda)$.
 \item ${\rm Ext}_{\mathcal{O}}(M,N)\cong{\rm Ext}_{\mathcal{O}}(N^\vee,M^\vee)$. See Sect.~3.1 of \cite{Humphreys2008} for the definition of the ${\rm Ext}_{\mathcal{O}}$ groups.
\end{itemize}
Evidently, $L(0)$ is the trivial $\mathfrak{g}_\C$-module. It is easy to see that there is an exact sequence
\begin{equation}\label{M0exactsequence1eq}
 0\longrightarrow L(2)\longrightarrow N(0)\longrightarrow L(0)\longrightarrow0.
\end{equation}
Since $N(0)$ is indecomposable, this sequence does not split. Applying the duality functor, we get another non-split exact sequence
\begin{equation}\label{M0exactsequence2eq}
 0\longrightarrow L(0)\longrightarrow N(0)^\vee\longrightarrow L(2)\longrightarrow0.
\end{equation}
It is an exercise to show that the sequence \eqref{M0exactsequence1eq} is the only non-trivial extension of $L(0)$ by $L(2)$; similarly for \eqref{M0exactsequence2eq}. The fact that ${\rm Ext}_{\mathcal{O}}(L(0),L(2))=1$ can also be seen by applying the functor $\Hom_{\mathcal{O}}(\underline{\,\phantom{x}},L(2))$ to \eqref{M0exactsequence1eq} and considering the resulting long exact sequence.
\subsection*{Proof of Proposition \ref{AAnfindecompprop}}
Before starting the proof of Proposition \ref{AAnfindecompprop}, let us comment on the relationship between $(\mathfrak{g},K)$-modules and $\mathfrak{g}_\C$-modules. Clearly, every $(\mathfrak{g},K)$-module is also a $\mathfrak{g}_\C$-module. Conversely, let $(\pi,V)$ be a $\mathfrak{g}_\C$-module all of whose weights are integral, and such that $V$ is the direct sum of its weight spaces. If $v\in V$ has weight $k$, i.e., if $\pi(H)v=kv$, then we define a $K$-action on $\C v$ by $\pi(r(\theta))v=e^{ik\theta}v$. Since $V$ is the direct sum of its weight spaces, this defines a $K$-action on all of $V$. One can verify that, with this $K$-action, $V$ becomes a $(\mathfrak{g},K)$-module. In particular, every module in category $\mathcal{O}$ is naturally a $(\mathfrak{g},K)$-module. The upshot is that in the following arguments we do not have to worry about the distinction between $(\mathfrak{g},K)$-modules and $\mathfrak{g}_\C$-modules.

With these comments in mind, it is clear that we have the following isomorphisms of $(\mathfrak{g},K)$-modules:
\begin{itemize}
 \item $\mathcal{F}_p\cong L(-p+1)$ for $p\geq1$;
 \item $\mathcal{D}_{p,+}\cong L(p+1)$ for $p\geq0$;
 \item $V_{\Phi_2}\cong N(0)^\vee$.
\end{itemize}
The third isomorphism follows by comparing the exact sequences \eqref{E2exactseqeq} and \eqref{M0exactsequence2eq}, observing the uniqueness comment made after \eqref{M0exactsequence2eq}.

We are now ready to prove Proposition \ref{AAnfindecompprop}. As we saw, the modular form $E_2$ gives rise to the submodule $V_{\Phi_2}$ of $\AA(\Gamma)_{\mathfrak{n}\text{-fin}}$. Every non-zero $f\in M_\ell(\Gamma)$ gives rise to a copy of $\mathcal{D}_{\ell-1,+}$ inside $\AA(\Gamma)_{\mathfrak{n}\text{-fin}}$. It is therefore clear that
\begin{equation}\label{AAnfindecomppropeq1}
  \AA(\Gamma)_{\mathfrak{n}\text{-fin}}\supset V_{\Phi_2}\:\oplus\;\bigoplus_{\ell=1}^\infty n_\ell\mathcal{D}_{\ell-1,+},\qquad n_\ell=\dim M_\ell(\Gamma).
\end{equation}
To prove the converse, observe first that
\begin{itemize}
 \item $\AA(\Gamma)_{\mathfrak{n}\text{-fin}}$ contains no negative weights;
 \item $\AA(\Gamma)_{\mathfrak{n}\text{-fin}}$ contains the weight $0$ exactly once, the corresponding weight space consisting of the constant functions.
\end{itemize}
Both statements follow from Lemma \ref{holkneglemma}. We define
$$
 \AA_{\leq k}(\Gamma)_{\mathfrak{n}\text{-fin}}=\sum_{\ell=0}^k\AA_\ell(\Gamma)_{\mathfrak{n}\text{-fin}},
$$
and let $\AA_{\langle\leq k\rangle}(\Gamma)_{\mathfrak{n}\text{-fin}}$ be the $(\mathfrak{g},K)$-module generated by elements of $\AA_{\leq k}(\Gamma)_{\mathfrak{n}\text{-fin}}$. Evidently,
\begin{equation}\label{AAnfindecomppropeq2}
 \AA_{\langle\leq k\rangle}(\Gamma)_{\mathfrak{n}\text{-fin}}\supset V_{\Phi_2}\:\oplus\;\bigoplus_{\ell=1}^k n_\ell\mathcal{D}_{\ell-1,+},\qquad n_\ell=\dim M_\ell(\Gamma).
\end{equation}
To prove equality in \eqref{AAnfindecomppropeq1}, it is enough to prove equality in \eqref{AAnfindecomppropeq2}.

It follows from the finite-dimensionality of the spaces $M_\ell(\Gamma)$ that $\AA_{\leq k}(\Gamma)_{\mathfrak{n}\text{-fin}}$ is finite-dimensional. Hence $\AA_{\langle\leq k\rangle}(\Gamma)_{\mathfrak{n}\text{-fin}}$ is finitely generated. This proves that $\AA_{\langle\leq k\rangle}(\Gamma)_{\mathfrak{n}\text{-fin}}$ is in category $\mathcal{O}$. By properties of this category, we may write
$$
 \AA_{\langle\leq k\rangle}(\Gamma)_{\mathfrak{n}\text{-fin}}=V_1\oplus\ldots\oplus V_n
$$
with indecomposable submodules. Let
\begin{equation}\label{Vifiltrationeq}
 0=V_{i,0}\subset V_{i,1}\subset\ldots\subset V_{i,n_i}=V_i
\end{equation}
be a filtration for $V_i$ such that $V_{i,j}/V_{i,j-1}\cong L(\lambda_{i,j})$ for some $\lambda_{i,j}\in\Z$. Since there are no negative weights, we have $\lambda_{i,j}\geq0$ for all $i$ and $j$.

Let $\chi_i$ be the character of $\mathcal{Z}$ such that $V_i=V_i(\chi_i)$; see \eqref{Mchidefeq}. We think of $\chi_i$ as a non-negative integer. Assume that $\chi_i>1$ or $\chi_i=0$. Then $\lambda_{i,j}=\chi_i+1$ for all $j$, since, among the $L(\lambda)$ with $\lambda\geq0$, only $L(\chi_i+1)$ has central character $\chi_i$. Now ${\rm Ext}_{\mathcal{O}}(L(\lambda),L(\lambda))=0$ for all $\lambda$ by Proposition 3.1 d) of \cite{Humphreys2008}. It follows that $V_i$ is a direct sum of copies of $L(\lambda_i)$, where $\lambda_i:=\chi_i+1$. Since $V_i$ is indecomposable, we must have $V_i=L(\lambda_i)$.

Now consider a $V_i$ with $\chi_i=1$ (i.e., $\chi_i=\chi_\varrho$). The only $L(\lambda)$ with $\lambda\geq0$ and this central character are $L(0)$ and $L(2)$. If $L(0)$ does not occur in $V_i$, then the same argument as above applies, and we see that $V_i=L(2)$. Assume that $L(0)$ does occur in $V_i$. Since the weight $0$ occurs exactly once in the entire space, there is exactly one $V_i$ with this property, and this $V_i$ contains $L(0)$ exactly once. Since the weight $0$ space consists of the constant functions, it appears as a subrepresentation in $V_i$. Hence, we may assume it occurs at the bottom of the filtration, i.e., $V_{i,1}=L(0)$. If $V_i$ would not contain any $L(2)$ subquotients, then $\AA_{\langle\leq k\rangle}(\Gamma)_{\mathfrak{n}\text{-fin}}$ would be completely reducible, which we know is not the case. Hence there is at least one $L(2)$ subquotient sitting on top of the $L(0)$. By \eqref{M0exactsequence2eq} and the remark following it, $V_{i,2}\cong N(0)^\vee$. Now
$$
 {\rm Ext}_{\mathcal{O}}(L(2),N(0)^\vee)\cong{\rm Ext}_{\mathcal{O}}(N(0),L(2))=0
$$
by Proposition 1.3 b) of \cite{Humphreys2008}. This means that there can be no further $L(2)$'s on top of the $N(0)^\vee$, and it follows that $V_i=N(0)^\vee$.

To summarize, we proved that, abstractly,
$$
 \AA_{\langle\leq k\rangle}(\Gamma)_{\mathfrak{n}\text{-fin}}=N(0)^\vee\:\oplus\;\bigoplus_{\lambda=1}^\infty m_\lambda L(\lambda)
$$
with non-negative integers $m_\lambda$ almost all of which are zero. A moment's consideration shows that $m_\lambda=0$ for $\lambda>k$. Since $L(\lambda)\cong\mathcal{D}_{\lambda-1,+}$ and $N(0)^\vee\cong V_{\Phi_2}$, and since we know $\mathcal{D}_{\lambda-1,+}$ cannot occur more than $\dim M_\lambda(\Gamma)$ times, comparison with \eqref{AAnfindecomppropeq2} shows that we must have equality in \eqref{AAnfindecomppropeq2}. This concludes the proof of Proposition \ref{AAnfindecompprop}.\qed
\subsection*{The Structure Theorem for all modular forms}
We can now provide an alternative proof of Theorem 5.2 of \cite{Shimura1987}.
\begin{theorem}[Structure theorem for all modular forms]\label{GL2noncuspstructuretheorem}
 Fix non-nega\-tive integers $k,p$ and a congruence subgroup $\Gamma$ of $\SL_2(\Z)$. Then:
 \begin{enumerate}
  \item $N_0^p(\Gamma)=\C$.
  \item If $k$ is even and $2\leq k<2+2p$, then
   \begin{equation}\label{GL2noncuspstructuretheoremeq1}
    N_k^p(\Gamma)=R^{(k-2)/2}(\C E_2)\;\oplus\bigoplus_{\substack{\ell\ge 1\\\ell\equiv k\bmod{2}\\k-2p\le\ell\le k}} R^{(k-\ell)/2}\left(M_\ell(\Gamma)\right).
   \end{equation}
  \item If $k$ is odd, or if $k$ is even and $k\geq2+2p$, then
   \begin{equation}\label{GL2noncuspstructuretheoremeq2}
    N_k^p(\Gamma)=\bigoplus_{\substack{\ell\ge 1\\\ell\equiv k\bmod{2}\\k-2p\le\ell\le k}} R^{(k-\ell)/2}\left(M_\ell(\Gamma)\right).
   \end{equation}
 \end{enumerate}
\end{theorem}
\begin{proof}
The proof is analogous to that of Theorem \ref{GL2cuspstructuretheorem}. Instead of Proposition \ref{AA0nfindecompprop}, we use Proposition \ref{AAnfindecompprop}. Observe that $R^{(k-2)/2}E_2$ is in $N_k^{k/2}(\Gamma)$, but not in $N_k^{k/2-1}(\Gamma)$, so in order for $E_2$ to contribute to $N_k^p(\Gamma)$ (for $k$ even) we must have $\frac k2\leq p$, or equivalently, $k<2+2p$.
\end{proof}

A simplified version of the Structure Theorem for all modular forms would be this: If $p<\frac{k-2}2$, then
\begin{equation}\label{GL2noncuspstructuretheoremeq3}
    N_k^p(\Gamma)=\bigoplus_{\substack{\ell\ge 1\\\ell\equiv k\bmod{2}\\k-2p\le\ell\le k}} R^{(k-\ell)/2}\left(M_\ell(\Gamma)\right).
\end{equation}
The hypothesis $p<\frac{k-2}2$ implies that, in the arguments in the proof of the theorem, we never ``reach down'' to weight $2$. Hence, the component $V_{\Phi_2}$ appearing in Proposition \ref{AAnfindecompprop} can be ignored.

\begin{corollary}[Structure theorem for non-cusp forms]\label{gl2eisensteinstrucutretheorem} Fix non-nega\-tive integers $k,p$ and a congruence subgroup $\Gamma$ of $\SL_2(\Z)$. Then:
 \begin{enumerate}
  \item $\E_0^p(\Gamma)=\C$.
  \item If $k$ is even and $2\leq k<2+2p$, then
   \begin{equation}\label{GL2eisstructuretheoremeq1}
    \E_k^p(\Gamma)=R^{(k-2)/2}(\C E_2)\;\oplus\bigoplus_{\substack{\ell\ge 1\\\ell\equiv k\bmod{2}\\k-2p\le\ell\le k}} R^{(k-\ell)/2}\left(E_\ell(\Gamma)\right).
   \end{equation}
  \item If $k$ is odd, or if $k$ is even and $k\geq2+2p$, then
   \begin{equation}\label{GL2eisstructuretheoremeq2}
    \E_k^p(\Gamma)=\bigoplus_{\substack{\ell\ge 1\\\ell\equiv k\bmod{2}\\k-2p\le\ell\le k}} R^{(k-\ell)/2}\left(E_\ell(\Gamma)\right).
   \end{equation}
 \end{enumerate}

\end{corollary}
\begin{proof}That the right side of each equation is contained in the left side is immediate from Lemma~\ref{preserveeis} and the fact that $\Phi_2$ (the automorphic form corresponding to $E_2$) lies in the orthogonal complement of the cusp forms. That the left side is contained in the right side follows from Theorem~\ref{GL2noncuspstructuretheorem}, the fact that $M_\ell(\Gamma)$ is the orthogonal sum of $S_\ell(\Gamma)$ and $E_\ell(\Gamma)$, and the fact that the $R^v$ maps preserve inner products up to a constant.
\end{proof}

\begin{corollary}\label{Nkporthogonalcorollary}
 Let $k$ and $p$ be non-negative integers. The space $N_k^p(\Gamma)$ is the orthogonal direct sum of $N_k^p(\Gamma)^\circ$ and $\E_k^p(\Gamma)$.
\end{corollary}
\begin{proof}
By the structure theorems, it is enough to prove the assertion for $p=0$. In this case the claim is that $M_k(\Gamma)$ is the orthogonal direct sum of $S_k(\Gamma)$ and $E_k(\Gamma)$. Clearly,
$$
 M_k(\Gamma)=S_k(\Gamma)\oplus S_k(\Gamma)^\perp\qquad\text{and}\qquad E_k(\Gamma)\subset S_k(\Gamma)^\perp.
$$
Hence, our task is to show that a non-zero element $f$ of $S_k(\Gamma)^\perp$ is orthogonal to all of $N_k(\Gamma)^\circ$. Let $\Phi_f$ be the function on $\SL_2(\R)$ corresponding to $f$. We will in fact show that $\Phi_f$ is orthogonal to any cusp form $\Psi$. We may assume that $\Psi$ generates an irreducible representation $\mathcal{D}_{\ell-1,+}$. Assume first that $\Psi$ has weight $\ell$, i.e., $\Psi$ is the lowest weight vector in $\mathcal{D}_{\ell-1,+}$. If $\ell\neq k$, then $\langle\Phi,\Psi\rangle=0$ since the weights do not match. If $\ell=k$, then $\langle\Phi,\Psi\rangle=0$ since $\Psi$ corresponds to an element of $S_k(\Gamma)$. Now assume that $\Psi$ has weight greater than $\ell$. Then $\Psi=R\Psi'$ for some $\Psi'\in\mathcal{D}_{\ell-1,+}$, and the general formula
$$
 \langle\Phi,R\Psi'\rangle+\langle L\Phi,\Psi'\rangle=0
$$
shows that $\langle\Phi,R\Psi'\rangle=0$, because $\Phi$ is a lowest weight vector. This concludes the proof.
\end{proof}
\begin{remark}
It is well-known that $E_k(\Gamma)=S_k(\Gamma)^\perp$ is spanned by the various weight $k$ holomorphic Eisenstein series on $\Gamma$.
\end{remark}

\bibliography{pullback}{}

\end{document}